\documentclass[12pt, reqno]{amsart}
\usepackage{amsmath, amsthm, amscd, amsfonts, amssymb, graphicx}
\usepackage[bookmarksnumbered, colorlinks, plainpages]{hyperref}
\input{mathrsfs.sty}

\def\!{\,!\,}

\newtheorem{theorem}{Theorem}[section]

\newtheorem{proposition}[theorem]{Proposition}
\newtheorem{corollary}[theorem]{Corollary}
\theoremstyle{definition}
\newtheorem{definition}[theorem]{Definition}
\newtheorem{example}[theorem]{Example}

\theoremstyle{remark}
\newtheorem{remark}[theorem]{Remark}

\numberwithin{equation}{section}

\begin{document}
\title{A Morita equivalence for Hilbert $C^*$-modules}
\author[M. Joi\c{t}a]{Maria Joi\c{t}a}
\address{Department of Mathematics, University of Bucharest, Bd. Regina
Elisabeta nr. 4-12, Bucharest, Romania.}
\email{mjoita@fmi.unibuc.ro}
\urladdr{http://sites.google.com/a/g.unibuc.ro/maria-joita/}
\author[M. S. Moslehian]{Mohammad Sal Moslehian}
\address{Department of Pure Mathematics, Center of Excellence in Analysis on
Algebraic Structures (CEAAS), Ferdowsi University of Mashhad, P. O. Box
1159, Mashhad 91775, Iran.}
\email{moslehian@ferdowsi.um.ac.ir and moslehian@member.ams.org}
\urladdr{http://profsite.um.ac.ir/~moslehian/}

\begin{abstract}
In this paper we introduce a notion of Morita equivalence for Hilbert $C^*$%
-modules in terms of the Morita equivalence of the algebras of compact
operators on Hilbert $C^*$-modules. We investigate some properties of the
new version of Morita equivalence and obtain some results. We then applied
our results to study the continuous actions of locally compact groups on
full Hilbert $C^*$-modules. We present an extension of Green's theorem in
the context of Hilbert $C^*$-modules as well.
\end{abstract}

\maketitle


\baselineskip=17pt



\section{Introduction}

The notion of a Hilbert $C^*$-module is a generalization of that of a
Hilbert space in which the inner product takes its values in a $C^*$-algebra
rather than in the field of complex numbers. Hilbert $C^*$-modules are
useful tools in $AW^*$-algebra theory, theory of operator algebras, operator
$K$-theory, Morita equivalence of $C^*$-algebras, group representation
theory and theory of operator spaces. The theory of Hilbert $C^*$-modules is
very interesting on its own right. If $E$ is a Hilbert $C^*$-module over a $%
C^{\ast }$-algebra ${\mathscr{A}}$, then we denote by $\mathcal{L}_{%
\mathscr{A}}(E)$ and $\mathcal{K}_{\mathscr{A}}(E)$, respectively, the $C^*$%
-algebra of all adjointable maps on $E$ and the $C^*$-algebra of all
``compact'' operators on $E$. The linking algebra associated to a Hilbert $%
C^*$-module\ $E$ over a $C^{\ast }$-algebra $\mathscr{A}$ is denoted by $%
\mathcal{L}(E).$ If $\mathscr{H}$ is an infinite-dimensional separable
Hilbert space, then $\mathbb{K}$ denotes the $C^*$-algebra of all compact
operators on $\mathscr{H}$.

The notion of (strong) Morita equivalence of $C^*$-algebras was first
introduced by Rieffel \cite{RIE}. Two $C^*$-algebras $\mathscr{A}$ and $%
\mathscr{B}$ are said to be Morita equivalent if there exists an $\mathscr{A}%
-\mathscr{B}$-imprimitivity bimodule, or equivalently, there exists a full
Hilbert $\mathscr{B}$-module $E$ such that $\mathscr{A}$ is isomorphic to
the $C^*$-algebra $\mathcal{K}_{\mathscr{B}}(E)$. There are other equivalent
definitions of Morita equivalence in the literature, see \cite{RW}. This
concept is weaker than the notion of $C^*$-isomorphism. It is an interesting
problem to study the properties of $C^*$-algebras being preserved under the
Morita equivalence. It is known that the Morita equivalence preserves $K$%
-theory and $K$-homology and several properties of $C^*$-algebras such as
type I-ness \cite{RIE}, nuclearity \cite{BEE} and simplicity \cite{RW}. Two
unital $C^*$-algebras are Morita equivalent if and only if they are Morita
equivalent as rings, cf. \cite{BEE}. Also two $C^*$-algebras are Morita
equivalent if and only if their minimal dense ideals are Morita equivalent,
cf. \cite{ARA}. A foundation of Morita equivalence theory for operator
algebras is established in \cite{BMP}. Muhly and Solel \cite{MS} defined a
notion of Morita equivalence for $C^*$-correspondences.

A Hilbert $\mathscr{A}$-module $E$ and a Hilbert $\mathscr{B}$-module $F$
are Morita equivalent in the sense of Skeide \cite{SKE11} if there exists a
Morita equivalence $M$ from $\mathscr{A}$ to $\mathscr{B}$ such that $%
E\otimes M=F$ (or $E=F\otimes M^*$). In \cite[Definition 3.4]{SKE1}, two
Hilbert $C^*$-module $E$ and $F$ are called stably Morita equivalent if $%
E\otimes \mathscr{H}$ and $F\otimes \mathscr{H}$ are Morita equivalent,
where $\mathscr{H}$ denotes any infinite-dimensional separable Hilbert
space. Two full Hilbert $C^*$-modules $E$ and $F$ are Morita equivalent in
the sense of Skeide if and only if the $C^*$-algebras $\mathcal{L}_{%
\mathscr{A}}(E)$ and $\mathcal{L}_{\mathscr{B}}(F)$ are bistrictly
isomorphic, and this if and only if the $C^*$-algebras $\mathcal{K}_{%
\mathscr{A}}(E)$ and $\mathcal{K}_{\mathscr{B}}(F)$ are isomorphic \cite[%
Corollaries 2.13, 2.14 and 2.16]{SKE1}.

If two $C^*$-algebras $\mathscr{A}$ and $\mathscr{B}$ are Morita equivalent
as Hilbert $C^*$-modules, then $\mathscr{A}$ and $\mathscr{B}$ are Morita
equivalent as $C^*$-algebras (in fact, $\mathscr{A}$ and $\mathscr{B}$ are
isomorphic). The converse implication is not true. So the notion of Morita
equivalence introduced by Skeide is stronger than the notion of Morita
equivalence by Rieffel.

In this paper we introduce a notion of Morita equivalence for Hilbert $C^*$%
-modules. It is defined as Morita equivalence of the algebras of compact
operators on Hilbert $C^*$-modules. This notion is weaker than that of
Skeide but under some ``standing'' countability hypotheses ($\sigma $-unital
$C^*$-algebras and countably generated modules) our definition coincides
with Skeide's definition of stable Morita equivalence. We investigate some
properties of the new version of Morita equivalence and obtain some results.
We then applied our results to study the continuous actions of locally
compact groups on full Hilbert $C^*$-modules. We present an extension of
Green's theorem in the context of Hilbert $C^*$-modules as well.

\section{Results}

We start this section with the following essential definition.

\begin{definition}
Two Hilbert $C^*$-modules $E$ and $F$, respectively, over $C^*$-algebras $%
\mathscr{A}$ and $\mathscr{B}$ are Morita equivalent, denoted by $%
E\backsim_{M} F$, if the $C^*$-algebras $\mathcal{K}_{\mathscr{A}}(E)$ and $%
\mathcal{K}_{\mathscr{B}}(F)$ are Morita equivalent.
\end{definition}

It is well known that any $C^*$-algebra $\mathscr{A}$ is a Hilbert $%
\mathscr{A}$-module in a natural way and the $C^*$-algebras $\mathscr{A}$
and $\mathcal{K}_{\mathscr{A}}(\mathscr{A})$ are isomorphic and so they are
Morita equivalent. Therefore, two $C^{*}$-algebras $\mathscr{A}$ and $%
\mathscr{B}$ are Morita equivalent as Hilbert $C^*$-modules if and only if
they are Morita equivalent as $C^{*}$-algebras.

Since the Morita equivalence of $C^*$-algebras is an equivalence relation,
the Morita equivalence $\backsim_M$ on Hilbert $C^*$-modules is an
equivalence relation too.

\begin{example}
Let $\mathscr{H}$ be a separable infinite dimensional Hilbert space. Then $%
\mathscr{H}$ $\backsim _{M}\mathbb{C}$ $\backsim _{M}\mathbb{K}$ as Hilbert $%
C^*$-modules, since the $C^*$-algebras $\mathbb{C}$ and $\mathbb{K}$ are
Morita equivalent.
\end{example}

A \textit{morphism of Hilbert $C^*$-modules} from a Hilbert $C^*$-module $E$
over $\mathscr{A}$ to a Hilbert $C^*$-module $F$ over $\mathscr{B}$ is a map
$\Phi :E\to F$ with the property that there is a $C^*$-morphism $\varphi :%
\mathscr{A}\to \mathscr{B}$ such that
\begin{equation*}
\left\langle \Phi \left( \xi _{1}\right),\Phi \left( \xi _{2}\right)
\right\rangle =\varphi \left( \left\langle \xi _{1},\xi _{2}\right\rangle
\right)
\end{equation*}%
for all $\xi _{1},\xi _{2}\in E$. Two Hilbert $C^*$-modules $E$ and $F$,
respectively, over $C^*$-algebras $\mathscr{A}$ and $\mathscr{B}$ are
\textit{isomorphic} if there is a bijective map $\Phi :E\to F$ such that $%
\Phi$ and $\Phi ^{-1}$ are morphisms of Hilbert $C^*$-modules.

\begin{proposition}
Let $E$ and $F$ be two Hilbert $C^*$-modules over $C^{\ast }$-algebras $%
\mathscr{A}$ and $\mathscr{B}$, respectively. If $E$ and $F$ are isomorphic,
then $E\backsim _{M}F$.
\end{proposition}

\begin{proof}
Since the Hilbert $C^*$-modules $E$ and $F$ are isomorphic, the $C^*$%
-algebras $\mathcal{K}_{\mathscr{A}}(E)$ and $\mathcal{K}_{\mathscr{B}}(F)$
are isomorphic \cite[Proposition 2.11]{1}, and so they are Morita
equivalent. Therefore $E\backsim _{M}F$.
\end{proof}

Given two Hilbert $C^*$-modules $E$ and $F$ over $C^*$-algebras $\mathscr{A}$
and $\mathscr{B}$, respectively, recall that the \textit{exterior tensor
product} $E\otimes F$ of $E$ and $F$ is a Hilbert $C^*$-module over the
injective tensor product $\mathscr{A}$ $\otimes $ $\mathscr{B}$ of $%
\mathscr{A}$ and $\mathscr{B}$ (see \cite{5}).

\begin{proposition}
Let $E_{1},E_{2},F_{1}$ and $F_{2}$ be four Hilbert $C^*$-modules over $C^*$%
-algebras $\mathscr{A}_{1,}\mathscr{A}_{2},\mathscr{B}_{1}$ and $\mathscr{B}%
_{2}$, respectively. If $E_{1}\backsim _{M}F_{1}$ and $E_{2}\backsim
_{M}F_{2}$, then $E_{1}\otimes E_{2}\backsim _{M}F_{1}\otimes F_{2}$.
\end{proposition}

\begin{proof}
From $E_{i}\backsim _{M}F_{i},$ $i=1,2$, we have $\mathcal{K}_{\mathscr{A}%
_{i}}(E_{i})$ $\backsim _{M}\mathcal{K}_{\mathscr{B}_{i}}(F_{i})\,\,(i=1,2)$%
, and then $\mathcal{K}_{\mathscr{A}_{1}}(E_{1})\otimes $ $\mathcal{K}_{%
\mathscr{A}_{2}}(E_{2})\backsim _{M}\mathcal{K}_{\mathscr{B}%
_{1}}(F_{1})\otimes \mathcal{K}_{\mathscr{B}_{2}}(F_{2})$. But the $C^*$%
-algebras $\mathcal{K}_{\mathscr{A}_{1}}(E_{1})\otimes $ $\mathcal{K}_{%
\mathscr{A}_{2}}(E_{2})$ and $\mathcal{K}_{\mathscr{A}_{1}\otimes \mathscr{A}%
_{2}}(E_{1}\otimes E_{2})$ are isomorphic as well as the $C^*$-algebras $%
\mathcal{K}_{\mathscr{B}_{1}}(F_{1})\otimes $ $\mathcal{K}_{\mathscr{B}%
_{2}}(F_{2})$ and $\mathcal{K}_{\mathscr{B}_{1}\otimes \mathscr{B}%
_{2}}(F_{1}\otimes F_{2})$ (see, for example, \cite[p. 57]{5} ). Therefore $%
\mathcal{K}_{\mathscr{A}_{1}\otimes \mathscr{A}_{2}}(E_{1}\otimes
E_{2})\backsim _{M}\mathcal{K}_{\mathscr{B}_{1}\otimes \mathscr{B}%
_{2}}(F_{1}\otimes F_{2})$.
\end{proof}

\begin{corollary}
Let $E$ be a Hilbert $C^*$-module. Then $E\backsim _{M}E\otimes\mathscr{H}%
\backsim _{M}E\otimes \mathbb{K}$.
\end{corollary}

\begin{proof}
From Example 2 and Proposition 4, we have $E\otimes \mathbb{C}\backsim
_{M}E\otimes \mathscr{H}\backsim _{M}E\otimes \mathbb{K}$.  Due to the
Hilbert $C^{\ast }$-modules $E$ and $E\otimes \mathbb{C}$ are isomorphic we
have $E\backsim _{M}E\otimes \mathbb{C}$.
\end{proof}

\begin{corollary}
Let $E$ and $F$ be two Hilbert $C^*$-modules over $C^{\ast }$-algebras $%
\mathscr{A}$ and $\mathscr{B}$, respectively. If the Hilbert $C^*$-modules $%
E\otimes \mathscr{H}$ and $F\otimes \mathscr{H}$ are isomorphic for some
separable Hilbert space $\mathscr{H}$, then $E\backsim _{M}F$.
\end{corollary}

\begin{proof}
If the Hilbert $C^*$-modules $E\otimes \mathscr{H}$ and $F\otimes\mathscr{H}$
are isomorphic, then, by Proposition 3, $E\otimes \mathscr{H}\backsim
_{M}F\otimes \mathscr{H}$. By Corollary 5, $E\backsim _{M}E\otimes\mathscr{H}
$ and $F\backsim _{M}F\otimes \mathscr{H}$. Hence $E\backsim _{M}F$.
\end{proof}

Let $E$ and $F$ be two Hilbert $C^*$-modules, respectively, over $C^{\ast }$%
-algebras $\mathscr{A}$ and $\mathscr{B}$ and let $\Phi: \mathscr{A}\to
\mathcal{K}_{\mathscr{B}}(F)$ be a $C^*$-morphism. Recall that the \textit{%
inner tensor product} $E\otimes _{\Phi }F$ of $E$ and $F$ corresponding to $%
\Phi $ is a Hilbert $C^*$-module over $\mathscr{B}$ (see \cite{5}).

\begin{proposition}
Let $E_{1},E_{2},F_{1}$ and $F_{2}$ be four Hilbert $C^*$-modules over $C^*$%
-algebras $\mathscr{A}_{1,}\mathscr{A}_{2},\mathscr{B}_{1}$ and $\mathscr{B}%
_{2}$, respectively. If $\Phi _{i}:\mathscr{A}_{i}\to \mathcal{K}_{%
\mathscr{B}_{i}}(F_{i})$\,\,($i=1,2$) are $C^*$-isomorphisms and $%
E_{1}\backsim _{M}E_{2}$, then $E_{1}\otimes _{\Phi _{1}}F_{1}\backsim
_{M}E_{2}\otimes _{\Phi _{2}}F_{2}$.
\end{proposition}

\begin{proof}
By \cite[Proposition 4.7]{5}, the $C^{\ast }$-algebras $\mathcal{K}_{%
\mathscr{A}_{i}}(E_{i})$ and $\mathcal{K}_{\mathscr{B}_{i}}(E_{i}\otimes
_{\Phi _{i}}F_{i})$ are isomorphic for $i=1,2$, and then $E_{i}\backsim
_{M}E_{i}\otimes _{\Phi _{i}}F_{i}$ for $i=1,2$. It follows from $%
E_{1}\backsim _{M}E_{2}$ that $E_{1}\otimes _{\Phi _{1}}F_{1}\backsim
_{M}E_{2}\otimes _{\Phi _{2}}F_{2}$.
\end{proof}

\begin{proposition}
Let $E$ and $F$ be two full Hilbert $C^*$-modules over $C^*$-algebras $%
\mathscr{A}$ and $\mathscr{B}$, respectively. Then the following assertions
are equivalent:

\begin{enumerate}
\item $E\backsim _{M}F;$

\item $\mathcal{L}(E)\backsim _{M}\mathcal{L}(F);$

\item $\mathscr{A}\backsim _{M}\mathscr{B}$.
\end{enumerate}
\end{proposition}

\begin{proof}
$\left( 1\right) \Leftrightarrow \left( 2\right) $ Since $E$ and $F$ are
full, $\mathcal{K}_{\mathscr{A}}(E)\backsim _{M}\mathcal{L}(E)$ and $%
\mathcal{K}_{\mathscr{B}}(F)\backsim _{M}\mathcal{L}(F)$. Therefore, $%
E\backsim _{M}F$ if and only if $\mathcal{L}(E)\backsim _{M}\mathcal{L}(F)$.%
\newline
For $\left( 2\right) \Leftrightarrow \left( 3\right)$ see \cite[Theorem 1.1]%
{2}.
\end{proof}

\begin{corollary}
Let $E$ be a full Hilbert $\mathscr{A}$-module. Then $E\backsim _{M}E\oplus %
\mathscr{A}\backsim _{M}\mathscr{A}\otimes \mathscr{H}\backsim _{M}E\otimes %
\mathscr{H}.$
\end{corollary}

Given a Hilbert $C^*$-module $E$ over $C^*$-algebra $\mathscr{A}$, the
vector space $\mathcal{L}_{\mathscr{A}}\left( \mathscr{A},E\right) $ of all
adjointable module morphisms from $\mathscr{A}$ into $E$ has a canonical
Hilbert $C^*$-module structure over the multiplier algebra $M(\mathscr{A})$
of $\mathscr{A}$, which is called the \textit{multiplier module }of $E$ (see
\cite{8}). The vector space $\mathcal{K}_{\mathscr{A}}(E,\mathscr{A})$ of
all compact operators from $E$ into $\mathscr{A}$ has a natural Hilbert $C^*$%
-module structure over $\mathcal{K}_{\mathscr{A}}(E)$, which is denoted by $%
E^*.$

\begin{theorem}
Let $E$ and $F$ be two Hilbert $C^*$-modules over $C^{\ast }$-algebras $%
\mathscr{A}$ and $\mathscr{B}$, respectively.

\begin{enumerate}
\item If the Hilbert $C^*$-modules $E\otimes \mathbb{K}$ and $F\otimes
\mathbb{K}$ are isomorphic, then $E\backsim _{M}F$.

\item If $E$ and $F$ are full countably generated, $E^*$ and $F^{\ast } $
are countably generated in their corresponding multiplier modules and $%
E\backsim _{M}F$, then the Hilbert $C^*$-modules $E\otimes \mathbb{K}$ and $%
F\otimes \mathbb{K}$ are isomorphic
\end{enumerate}
\end{theorem}

\begin{proof}
$\left( 1\right) $ It follows from Proposition 3 and Corollary 5.

$\left( 2\right) $ If $E\backsim _{M}F$, then $\mathcal{K}_{\mathscr{A}%
}(E)\backsim_{M}\mathcal{K}_{\mathscr{B}}\left( F\right)$. Since $E$ and $F$
are countably generated, the $C^*$-algebras $\mathcal{K}_{\mathscr{A}}(E)$
and $\mathcal{K}_{\mathscr{B}}\left( F\right) $ are $\sigma$-unital, i.e
possess countably approximate unit (see \cite[Proposition 6.7]{5}), and by
\cite[ Theorem 1.2]{2}, the $C^*$ -algebras $\mathcal{K}_{\mathscr{A}%
}(E)\otimes \mathbb{K}$ and $\mathcal{K}_{\mathscr{B}}\left( F\right)
\otimes \mathbb{K}$ are isomorphic.

On the other hand, since $E$ and $E^*$ are countably generated in in their
corresponding multiplier modules, the Hilbert $C^*$-modules $\mathscr{A}%
\otimes \mathbb{K}$ and $E\otimes \mathbb{K}$ are unitarily equivalent (see
\cite[Proposition 3.1]{3}) and then the $C^*$-algebras $\mathscr{A}\otimes
\mathbb{K}$ and $\mathcal{K}_{\mathscr{A}}(E)\otimes \mathbb{K}$ are
isomorphic. In the same manner, we deduce that the $C^*$-algebras $%
\mathscr{B}\otimes \mathbb{K}$ and $\mathcal{K}_{\mathscr{B}}\left( F\right)
\otimes \mathbb{K}$ are isomorphic. Therefore, the $C^*$-algebras $%
\mathscr{A}\otimes \mathbb{K}$ and $\mathscr{B}\otimes \mathbb{K}$ are
isomorphic and so they are isomorphic as Hilbert $C^*$-modules. From these
facts, we conclude that the Hilbert $C^*$-modules $E\otimes \mathbb{K}$ and $%
F\otimes \mathbb{K}$ are isomorphic.
\end{proof}

\begin{remark}
If $\mathscr{A}$ is a $\sigma$-unital $C^*$-algebra, then $\mathscr{A}$ is a
full countably generated Hilbert $C^*$-module over $\mathscr{A}$ and since $%
\mathcal{K}_{\mathscr{A}}(\mathscr{A}, \mathscr{A})$ is isomorphic to $%
\mathscr{A}$, $\mathscr{A}^*$ is countably generated in its multiplier
module. Therefore, Theorem 10 extends \cite[Theorem 1.2]{2}.
\end{remark}

The next result is an extension of \cite[Theorem 1.2]{2}.

\begin{corollary}
Let $E$ and $F$ be two full countably generated Hilbert $C^*$-modules over
commutative $C^*$-algebras $\mathscr{A}$ and $\mathscr{B}$, respectively.
Then $E\backsim _{M}F$ if and only if the Hilbert $C^*$-modules $E\otimes
\mathbb{K}$ and $F\otimes \mathbb{K}$ are isomorphic.
\end{corollary}

\begin{proof}
It follows from Theorem 10 and \cite[Corollary 3.7]{3}.
\end{proof}

\begin{corollary}
Let $\mathscr{A}$ and $\mathscr{B}$ be two $C^*$-algebras. Then $\mathscr{A}$
$\backsim _{M}$ $\mathscr{B}$ if and only if there are two countably
generated full Hilbert $C^*$-modules $E$ and $F$, respectively, over $%
\mathscr{A}$ and $\mathscr{B}$ such that $E^*$and $F^*$ are countably
generated in their corresponding multiplier modules and the Hilbert $C^*$%
-modules $E\otimes \mathbb{K}$ and $F\otimes \mathbb{K}$ are isomorphic.
\end{corollary}

\begin{theorem}
Let $E$ and $F$ be two full countably generated Hilbert $C^*$-modules over $%
\sigma$-unital $C^*$-algebras $\mathscr{A}$ and $\mathscr{B}$, respectively.
Then the following statements are equivalent:

\begin{enumerate}
\item $E$ and $F$ are Morita equivalent;

\item $E\otimes \mathbb{K}$ and $F\otimes \mathbb{K}$ are isomorphic;

\item $\mathcal{K}_{\mathscr{A}}(E)$ and $\mathcal{K}_{\mathscr{B}}(F)$ are
stably isomorphic;

\item $\mathcal{K}_{\mathscr{A}}(E)$ and $\mathcal{K}_{\mathscr{B}}(F)$ are
Morita equivalent;

\item $\mathscr{A}$ and $\mathscr{B}$ are Morita equivalent;

\item $\mathscr{A}$ and $\mathscr{B}$ are stably isomorphic.
\end{enumerate}
\end{theorem}

\begin{proof}
Since $E$ and $F$ are full Hilbert $C^*$-modules over $\sigma$-unital $C^*$%
-algebras, the Hilbert $C^*$-modules $E^*$ and $F^*$ are countably generated
and so they are countably generated in their corresponding multiplier
modules. Therefore equivalence of (1) and (2) is nothing than Theorem 10.
The equivalence of (5) and (6) is \cite[Theorem 1.2]{2}, and the equivalence
(1) and (5) is Proposition 8. Since $E$ and $F$ are countably generated, the
$C^*$-algebras $\mathcal{K}_{\mathscr{A}}(E)$ and $\mathcal{K}_{\mathscr{B}%
}(F)$ are $\sigma$-unital, and then the equivalence of (3) and (4) is \cite[%
Theorem 1.2]{2}. The equivalence of (1) and (4) is directly deduced from
Definition 1.
\end{proof}

\begin{remark}
By \cite[Theorem 3.5]{SKE1}, two full countably generated Hilbert $C^*$%
-modules $E$ and $F$ over $\sigma$-unital $C^*$-algebras are stably Morita
equivalent in the sense of Skeide if and only if they are modules over
Morita equivalent $C^*$-algebras. So, in this case, the notion of Morita
equivalence introduced in this note coincides with the notion of stable
Morita equivalence introduced by Skeide \cite{SKE11}.
\end{remark}

\begin{corollary}
Let $E$ and $F$ be two full countably generated Hilbert $C^*$-modules over $%
\sigma$-unital $C^*$-algebras $\mathscr{A}$ and $\mathscr{B}$, respectively.
Then $E\ $and $F$ are stably Morita equivalent in the sense of Skeide \cite%
{SKE1} if and only if the Hilbert $C^*$-modules $E\otimes \mathbb{K}$ and $%
F\otimes \mathbb{K}$ are isomorphic.
\end{corollary}

\begin{proof}
It follows from Theorem 14 and \cite[Theorem 3.5 (2)]{SKE1}.
\end{proof}

\section{Applications}

\bigskip Let $E$ be a full Hilbert $\mathscr{A}$-module and let $G$ be a
locally compact group. A \textit{continuous action} of $G$ on $E$ is a group
homomorphism $\eta $ from $G$ to Aut($E$), the group of all isomorphisms of
Hilbert $C^{\ast }$-modules from $E$ on $E$, such that the map $t\mapsto
\eta _{t}\left( x\right) $ from $G$ to $E$ is continuous for each $x\in E$.
Any continuous action $\eta $ of $G$ on $E$ induces a continuous action $%
\alpha ^{\eta }$ of $G$ on $\mathscr{A}$ by $\alpha _{g}^{\eta }\left(
\left\langle x,y\right\rangle \right) =\left\langle \eta _{g}\left( x\right)
,\eta _{g}\left( y\right) \right\rangle $ for all $x,y\in E$ and $g\in G$.
The linear space $C_{c}(G,E)$ of all continuous functions from $G$ to $E$
with compact support has a pre-Hilbert $G\times _{\alpha ^{\eta }}\mathscr{A}
$-module structure with the action of $G\times _{\alpha ^{\eta }}\mathscr{A}$
on $C_{c}(G,E)$ given by
\begin{equation*}
\left( \widehat{x}f\right) \left( s\right) =_{G}\widehat{x}%
\left( t\right) \alpha _{t}^{\eta }\left( f\left( t^{-1}s\right) \right) dt
\end{equation*}%
for all $\widehat{x}\in C_{c}(G,X)$ and all $f\in C_{c}(G,\mathscr{A})$ and
the inner product given by
\begin{equation*}
\left\langle \widehat{x},\widehat{y}\right\rangle \left( s\right)
=_{G}\alpha _{t^{-1}}^{\eta }\left( \left\langle \widehat{x}(t),%
\widehat{y}\left( ts\right) \right\rangle \right) dt.
\end{equation*}%
The \textit{crossed product of $E$ by $\eta $}, denoted by $G\times _{\eta }E
$, is the Hilbert $G\times _{\alpha ^{\eta }}\mathscr{A}$-module obtained by
the completion of the pre-Hilbert $G\times _{\alpha ^{\eta }}\mathscr{A}$%
-module $C_{c}(G,E)$ (see, for example, \cite{4}).

A continuous action $\eta $ of $G$ on $E$ induces a continuous action $\beta
^{\eta }$ of $G$ on $\mathcal{K}_{\mathscr{A}}\left( E\right) $ given by $%
\beta _{g}^{\eta }\left( \theta _{x,y}\right) =\theta _{\eta _{g}\left(
x\right),\eta _{g}\left( y\right) }$ and a continuous action $\gamma ^{\eta }
$ of $G$ on the linking algebra $\mathcal{L}\left( E\right) $ given by $%
\gamma _{g}^{\eta }\left( \theta _{a\oplus x,b\oplus y}\right) =\theta
_{\alpha _{g}^{\eta }\left( a\right) \oplus \eta _{g}\left( x\right),\alpha
_{g}^{\eta }\left( b\right) \oplus \eta _{g}\left( y\right) }$.

Recall that two continuous actions $\alpha $ and $\beta $ of a locally
compact group $G$, respectively, on the $C^*$-algebras $\mathscr{A}$ and $%
\mathscr{B}$ are Morita equivalent if there is a full Hilbert $C^*$-module $E
$ over $\mathscr{A}$ and a continuous action $\eta $ of $G$ on $E$ such that
the $C^*$-algebras $\mathcal{K}_{\mathscr{A}}\left( E\right) $ and $%
\mathscr{B}$ are isomorphic, $\alpha =\alpha ^{\eta }$ and $\varphi \circ
\beta =\beta ^{\eta }\circ \varphi $, where $\varphi $ is an isomorphism
from $\mathscr{B}$ onto $\mathcal{K}_{\mathscr{A}}\left( E\right) $.

\begin{definition}
Two continuous actions $\eta $ and $\mu $ of a locally compact group $G$ on
the full Hilbert $C^*$-modules $E$ and $F$ over $C^*$-algebras $\mathscr{A}$
and $\mathscr{B}$, respectively, are Morita equivalent if the actions $\beta
^{\eta }$ and $\beta ^{\mu }$ of $G$, respectively, on $\mathcal{K}_{%
\mathscr{A}}\left( E\right)$ and $\mathcal{K}_{\mathscr{B}}\left( F\right)$
are Morita equivalent.
\end{definition}

The following proposition extends \cite[Theorem 1 ]{CMW} and \cite[Theorem,
p.299]{COM} in the context of Hilbert $C^*$-modules.

\begin{proposition}
Let $\eta $ and $\mu $ be two continuous actions of a locally compact group $%
G$ on the full Hilbert $C^*$-modules $E$ and $F$, respectively. If $\eta $
and $\mu $ are Morita equivalent, then the Hilbert $C^*$-modules $G\times
_{\eta }E$ and $G\times _{\mu }F$ are Morita equivalent.
\end{proposition}

\begin{proof}
If the actions $\eta $ and $\mu $ are Morita equivalent, then the actions $%
\gamma ^{\eta }$ and $\gamma ^{\mu }$ are Morita equivalent, since the
actions $\beta ^{\eta }$ and $\gamma ^{\eta }\ $are Morita equivalent, the
actions $\beta ^{\mu }$ and $\gamma ^{\mu }$ are Morita equivalent (see \cite%
[p. 297]{COM}) and the relation of Morita equivalence of actions of groups
on $C^*$-algebras is an equivalence relation. Then, by \cite[Theorem, p. 299]%
{COM}, the $C^*$-algebras $G\times _{\gamma ^{\eta }}\mathcal{L} (E) $ and $%
G\times _{\gamma ^{\mu }}\mathcal{L}(F)$ are Morita equivalent. On the other
hand, the $C^*$-algebras $G\times _{\gamma ^{\eta }} \mathcal{L}(E)$ and $%
\mathcal{L}(G\times _{\eta }E)$ are isomorphic as well as the $C^*$-algebras
$G\times _{\gamma ^{\mu }}\mathcal{L}(F)\ $and $\mathcal{L}(G\times _{\mu }F)
$ (see the proof of \cite[Theorem 4.1]{EKQR}). Hence the $C^*$-algebras $%
\mathcal{L}(G\times _{\eta }E)$ and $\mathcal{L}(G\times _{\mu }F)$ are
Morita equivalent. From these facts and Proposition 8, we conclude that the
Hilbert $C^*$-modules $G\times _{\eta }E$ and $G\times _{\mu }F$ are Morita
equivalent.
\end{proof}

The vector space $C_{0}\left( G,E\right) $ of all continuous functions
vanishing at infinity from $G$ to $E$ has a canonical Hilbert $C^*$-module
structure over $C^*$-algebra $C_{0}(G,$ $\mathscr{A})$, which can be
identified to $C_{0}(G)\otimes \mathscr{A}$. Moreover, $C_{0}\left(
G,E\right) $ is full and can be identified to $C_{0}(G)\otimes E.$

The following theorem extends Green's theorem to the context of Hilbert $C^*$%
-modules.

\begin{theorem}
Let $G$ be a locally compact group, $G_{0}$ a closed subgroup of $G$, $E$ a
full Hilbert $\mathscr{A}$-module, $\eta $ a continuous action of $G$ on $E$
and $\sigma ^{\eta }$ a continuous action of $G$ on the Hilbert $%
C_{0}(G/G_{0},\mathscr{A})$-module $C_{0}(G/G_{0},E)$ defined by
\begin{equation*}
\sigma _{t}^{\eta }(f)(sG_{0})=\eta _{t}(f(t^{-1}sG_{0})).
\end{equation*}%
Then the Hilbert $C^*$-modules $G\times _{\sigma ^{\eta }}C_{0}(G/G_{0},E)$
and $G_{0}\times _{\eta |_{G_{0}}}E$ are Morita equivalent.
\end{theorem}

\begin{proof}
Since the linking algebra of $C_{0}(G/G_{0},E)$ can be identified to $%
C_{0}(G/G_{0},\mathcal{L}(E))$, the action of $G$ on $\mathcal{L}%
(C_{0}(G/G_{0},E))$ induced by $\sigma ^{\eta }$ can be identified to the
action $\sigma ^{\gamma ^{\eta }}$ of $G$ on $C_{0}(G/G_{0},\mathcal{L}(E))$
given by $\sigma _{g}^{\gamma ^{\eta }}\left( f\right) \left( sG_{0}\right)
=\gamma _{g}^{\eta }\left( f\left( g^{-1}sG_{0}\right) \right) $. Therefore
the $C^*$-algebras $\mathcal{L}(G\times _{\sigma ^{\eta }}C_{0}(G/G_{0},E))$
and $G\times _{\sigma ^{\gamma ^{\eta }}}C_{0}(G/G_{0},\mathcal{L(}E))$ are
isomorphic (see the proof of \cite[Theorem 4.1]{EKQR}).

Clearly $\gamma ^{\eta }|_{G_{0}}=\gamma ^{\eta |_{G_{0}}}$, so , by \cite[%
Theorem 4.21]{RW}, the $C^*$-algebra $G_{0}\times _{\gamma ^{\eta }|_{G_{0}}}%
\mathcal{L}\left( E\right) $ $\ $is Morita equivalent to the $C^*$-algebra $%
G\times _{\sigma ^{\gamma ^{\eta }}}C_{0}(G/G_{0},\mathcal{L}\left( E\right)
)$ and is isomorphic to $\mathcal{L}\left( G_{0}\times _{\eta
|_{G_{0}}}E\right)$ (see the proof of \cite[Theorem 4.1]{EKQR}). Hence the $%
C^*$-algebras $G\times _{\sigma ^{\gamma ^{\eta }}}C_{0}(G/G_{0},\mathcal{L}%
\left( E\right) )$ and $\mathcal{L}\left( G_{0}\times _{\eta
|_{G_{0}}}E\right) $ are Morita equivalent. We have therefore showed
that the $C^*$-algebras $\mathcal{L}(G\times _{\sigma ^{\eta
}}C_{0}(G/G_{0},E))$ and $\mathcal{L}\left( G_{0}\times _{\eta
|_{G_{0}}}E\right) $ are Morita equivalent, whence, by Proposition
8, $G\times _{\sigma ^{\eta }}C_{0}(G/G_{0},E)$ $\backsim
_{M}G_{0}\times _{\eta |_{G_{0}}}E$.
\end{proof}


\begin{thebibliography}{99}
\bibitem{ARA} P. Ara, \textit{Morita equivalence and Pedersen ideals}, Proc.
Amer. Math. Soc. 129 (2001), no. 4, 1041--1049.

\bibitem{1} D. Baki\'{c} and B. Gulj\u{a}s, \textit{On a class of module
maps of Hilbert $C^*$-modules}, Math. Commun. 7 (2002), 177--192.

\bibitem{BEE} W. Beer, \textit{On Morita equivalence of nuclear $C^*$%
-algebras}, J. Pure Appl. Algebra 26 (1982), no. 3, 249--267.

\bibitem{BMP} D. P. Blecher, P. S. Muhly and V. I. Paulsen, \textit{%
Categories of operator modules (Morita equivalence and projective modules)},
Mem. Amer. Math. Soc. 143 (2000), no. 681.

\bibitem{2} L. G. Brown, P. Green and M. A. Rieffel, \textit{Stable
isomorphism and strong Morita equivalence of $C^*$-algebras}, Pacific J.
Math. 71 (1977), no. 2, 349--363.

\bibitem{COM} F. Combes, \textit{Crossed products and Morita equivalence},
Proc. London Math. Soc. 49 (1984), 289--306.

\bibitem{CMW} R. E. Curto, S. P. Muhly and D. P. Williams, \textit{Cross
products of strongly Morita equivalent $C^*$-algebras}, Proc. Amer. Math.
Soc. 90 (1984), 4, 528--530.

\bibitem{EKQR} S. Echterhoff, S. Kaliszewski, J. Quigg and I. Raeburn,
\textit{Naturality and induced representations}, Bull. Austral. Math. Soc.
61 (2000), 415--438.

\bibitem{3} M. Joi\c{t}a, \textit{A note on countably generated Hilbert
modules}, Results Math. 55 (2009), 1-2, 101-109

\bibitem{4} M. Kusuda, \textit{Duality for crossed products of Hilbert $C^*$%
-modules}, J. Operator Theory 60 (2008), 85--112.

\bibitem{5} E. C. Lance, \textit{Hilbert $C^*$-modules}, London Mathematical
Society Lecture Note Series, vol. 210, Cambridge University Press,
Cambridge, 1995, A toolkit for operator algebraists.

\bibitem{MS} P. S. Muhly and B. Solel, \textit{On the Morita equivalence of
tensor algebras}, Proc. London Math. Soc. (3)81 (2000), no. 1, 113--168.

\bibitem{8} I. Raeburn and S. J. Thompson, \textit{Countably generated
Hilbert modules, the Kasparov stabilisation theorem, and frames with Hilbert
modules}, Proc. Amer. Math. Soc. 131 (2003), no. 5, 1557--1564.

\bibitem{RW} I. Raeburn and D. P. Williams, \textit{Morita equivalence and
continuous-trace $C^*$-algebras}, Mathematical Surveys and Monographs, 60.
American Mathematical Society, Providence, RI, 1998.

\bibitem{RIE} M. A. Rieffel, \textit{Induced representations of $C^*$%
-algebras}, Advances in Math. 13 (1974), 176--257.

\bibitem{SKE11} M. Skeide, \textit{Unit vectors, Morita equivalence and
endomorphisms}, Publ. Res. Inst. Math. Sci. 45 (2009), no. 2, 475--518.

\bibitem{SKE1} M. Skeide, \textit{Classification of E\_0--Semigroups by
Product Systems}, arXiv:0901.1798v2.
\end{thebibliography}
\end{document}